\documentclass[12pt]{article}

\usepackage[margin=1in]{geometry}
\usepackage{amsmath}               
\usepackage{amsfonts}              
\usepackage{amsthm}
\usepackage{amssymb}
\usepackage[mathscr]{euscript}
\usepackage{dsfont}

\usepackage{graphicx,xcolor}
\usepackage{color}
\usepackage{transparent}

\usepackage[utf8]{inputenc}
\usepackage[T1]{fontenc}
\usepackage{lmodern}
\usepackage{amsmath}
\usepackage{bbm}

\newtheorem{thm}{Theorem}
\newtheorem{lem}[thm]{Lemma}
\newtheorem{prop}[thm]{Proposition}

\newtheorem{defi}[thm]{Definition}

\newcommand{\RR}{\mathbb{R}}      
\newcommand{\ZZ}{\mathbb{Z}}      
\newcommand{\NN}{\mathbb{N}}      
\newcommand{\CC}{\mathbb{C}}       
\newcommand{\PSL}{\mathrm{PSL}}
\newcommand{\R}{\mathrm{Re}\,}
\newcommand{\I}{\mathrm{Im}\,}
\newcommand{\HH}{\mathbb{H}}      
   
\newcommand{\tr}{\mathrm{tr}}   
\newcommand{\SL}{\mathrm{SL}}  

\begin{document}

\date{}
\title{Hyperbolic Surfaces with Arbitrarily Small Spectral Gap}
\author{Louis Soares}
\maketitle

\begin{abstract}
Let $ X = \Gamma\setminus \HH $ be a non-elementary geometrically finite hyperbolic surface and let $ \delta $ denote the Hausdorff dimension of the limit set $ \Lambda(\Gamma) $. We prove that for every $ \varepsilon > 0 $ the surface $ X $ admits a finite cover $ X' $ such that the Selberg zeta function associated to $ X' $ has a zero $ s\neq \delta $ with $ \vert \delta - s\vert < \varepsilon  $.

For $ \delta > \frac{1}{2} $ we exploit the combinatorial interpretation of spectral gap in terms of expander graphs.
For $ \delta \leq \frac{1}{2} $ the proof is based on the thermodynamic formalism approach for L-functions associated to hyperbolic surfaces and an analogue of the Artin-Takagi formula for these L-functions.
\end{abstract}

\section{Introduction and Results}

The spectral theory of compact and finite-area hyperbolic surfaces has a long history going back to Selberg \cite{Selberg}, Buser \cite{Buser} and Venkov \cite{Venkov} and is by now a classical subject. A spectral theory of infinite-area hyperbolic surfaces has appeared only recently, but the character of this theory is changed by the assumption of infinite area. Indeed, for infinite-area surfaces $ X $ the $ L^{2} $-spectrum of the $ \Delta_{X} $ is mostly continuous. (Here $ \Delta_{X} $ denotes the positive Laplacian on $ X $.) However there is a natural replacement for the missing spectrum: the \textit{resonances} of $ X $. These arise as poles of a meromorphic continuation of the resolvent of $ \Delta_{X} $ and the spectral theory of infinite-area hyperbolic surfaces is based on the existence of such a meromorphic continuation. A good reference on this subject is the book of Borthwick \cite{Borthwick}.

Let $ \mathbb{H} $ denote the Poincaré upper half-plane
$$ \mathbb{H} = \{ x+iy \in \mathbb{C} :  x\in \mathbb{R}, y>0\} $$
endowed with the hyperbolic metric
$$ ds^{2} = \frac{dx^{2}+dy^{2}}{y^{2}}. $$
The group of isometries of $ \mathbb{H} $ is isomorphic to $ \mathrm{PSL}_{2}(\mathbb{R}) = \mathrm{SL}_{2}(\RR)/\{ \pm I\} $, which acts on the hyperbolic plane by Möbius transformations
$$ z\mapsto \frac{az +b}{cz+ d}, \quad \begin{pmatrix}
a & b \\
c & d 
\end{pmatrix} \in \mathrm{SL}_{2}(\mathbb{R}). $$
Any hyperbolic surface $ X $ is isometric to a quotient $ X=\Gamma\setminus \HH $ by a Fuchsian group $ \Gamma $, that is, a discrete subgroup of $ \PSL_{2}(\RR). $ We will always assume that $ \Gamma $ is a non-elementary, finitely generated Fuchsian group without elliptic elements. Let $ \delta $ be the Hausdorff dimension of the limit set $ \Lambda(\Gamma) $ of $ \Gamma $. The spectrum of $ \Delta_{X} $ on $ L^{2}(X) $ was described by Lax-Phillipps \cite{LaxPhil}, \cite{LaxPhillips} and Patterson \cite{Patt}:

\begin{enumerate}
\item[(i)] The continuous spectrum is the interval $ [\frac{1}{4}, \infty), $
\item[(ii)] If $ X $ has infinite area, there are no embedded eigenvalues inside $ [\frac{1}{4}, \infty), $
\item[(iii)] The pure point spectrum (which we denote by $ \Omega(X) $) is empty if $ \delta \leq \frac{1}{2} $. If $ \delta > \frac{1}{2} $, the number $ \# \Omega(X) \cap [0, \frac{1}{4}) $ of eigenvalues in $ [0, \frac{1}{4}) $ is finite and the bottom of the spectrum is given by
$$ \lambda_{0}(\Gamma) := \inf \Omega(X) = \delta(1-\delta). $$
\end{enumerate}
The resolvent 
$$ R_{X}(s) := (\Delta_{X}-s(1-s))^{-1} : L^{2}(X) \to L^{2}(X) $$
defines a bounded operator for $ \R s > \frac{1}{2} $, except for a finite number of poles $ s\in (\frac{1}{2}, 1] $ corresponding to the $ L^{2} $-eigenvalues $ \lambda = s(1-s)\in [0, \frac{1}{4}) $. Mazzeo and Melrose \cite{MazzeoMelrose} proved that $ R_{X}(s) $
has a meromorphic continuation to the entire complex plane, as an operator from $ C_{c}^{\infty}(X) $ to $ C^{\infty}(X) $. By the result of Patterson \cite{Patt} and Sullivan \cite{Sull}, $ R_{X}(s) $ is analytic for $ \R s > \delta $. Patterson \cite{Patt2} proved that $ R_{X}(s) $ has a simple pole at $ s=\delta $ and no other poles on the line $ \R s = \delta. $ Let $ \mathcal{R}_{X} $ denote the set of poles of the resolvent $ R_{X}(s) $. We define the \textit{spectral gap} of $ X $ as the quantity
$$ \mathrm{gap}(X):= \inf_{\zeta\in \mathcal{R}_{X}-\{ \delta \}} ( \delta-\mathrm{Re}\, \zeta). $$

For $ \delta > \frac{1}{2} $ the spectral gap is positive, since in this case all the resonances $ s\in \mathcal{R}_{X} $ with $ \R s > \frac{1}{2} $ correspond to the eigenvalues $ \lambda = s(1-s)\in [0, \frac{1}{4}) $ and there are only finitely many of them. Now suppose $ \delta \leq \frac{1}{2} $. In this case there is no discrete $ L^{2} $-spectrum at all and the existence of a non-zero spectral gap is less obvious. Nevertheless, Naud \cite{Naud} showed that $ \mathrm{gap}(X) > 0 $ for all non-elementary, geometrically finite hyperbolic surfaces $ X = \Gamma\setminus \HH $. However the size of the gap in Naud's proof is hardly explicit. The aim of this paper is to show that
$$ \mathrm{gap}'(X) := \inf_{\zeta\in \mathcal{R}_{X}-\{ \delta \}} \vert\delta-  \zeta\vert \geq  \mathrm{gap}(X) $$ 
can be made arbitrarily small by passing to finite covers. Our first theorem is the following:

\begin{thm}\label{First Main Theorem}
Let $ \Gamma $ be a non-elementary finitely generated Fuchsian group without elliptic elements. Then there exists a sequence $ (\Gamma_{N})_{N\in \NN} $ of finite-index normal subgroups of $ \Gamma $ satisfying
$$ \underset{N\to \infty}{\lim} \mathrm{gap}'(X_{N})=0, $$
where $ X_{N} = \Gamma_{N}\setminus \HH $ is the associated cover of $ X = \Gamma\setminus \HH. $
\end{thm}

Now suppose in addition that $ \Gamma $ does not contain non-trivial parabolic elements or equivalently, that the quotient $ X= \Gamma\setminus \HH $ has no \textit{cusps}. Then by the work of Button \cite{Button}, $ \Gamma $ is a \textit{Schottky} group. Thus we call $ X $ a Schottky surface and in this case the following quantitative version of Theorem \ref{First Main Theorem} can be given:

\begin{thm}\label{main theorem}
Let $ X = \Gamma\setminus \HH $ be a Schottky surface. Then there exists a sequence $ (\Gamma_{N})_{N\in \NN} $ of finite-index normal subgroups of $ \Gamma $ satisfying the following property: for every $ \varepsilon > 0 $ there exists a constant $ c=c(\varepsilon, X) > 0 $ such that for every $ N $ we have 
$$ \# \left\{ \zeta\in \mathcal{R}_{X_{N}} : \vert \delta-\zeta \vert < \varepsilon \right\} \geq c [\Gamma:\Gamma_{N}], $$
where the resonances are counted with multiplicity.
\end{thm}

Note that Theorem \ref{First Main Theorem}
follows from Theorem \ref{main theorem} when $ X $ is a Schottky surface.

In Section \ref{construction} we give an explicit construction of the subgroups $ \Gamma_{N} $. It should be pointed out that the existence of finite coverings of compact hyperbolic surfaces giving arbitrarily small first eigenvalue was already known. We refer the reader to the book of Bergeron \cite{Bergeron}, Chapter 3 for the proof.

Let us recall the definition of the Selberg zeta function (for more background see \cite[Chapter~2]{Borthwick}). There is a one-to-one correspondence between the closed, oriented, primitive geodesics of a hyperbolic surface $ X = \Gamma\setminus\HH $ and the set $ \overline{\Gamma}_{p} $ of conjugacy classes of primitive hyperbolic elements of $ \Gamma $. The length of the closed geodesic corresponding to the conjugacy class $ \overline{\gamma} $ is given by the displacement length $ \ell(\gamma) $ of the element $ \gamma $.  The Selberg zeta function $ Z_{\Gamma} $ is defined by the Euler product
\begin{equation}\label{prodzeta}
Z_{\Gamma}(s) := \prod_{\overline{\gamma}\in \overline{\Gamma}_{p}} \prod_{k=0}^{\infty} \left( 1-e^{-(s+k)\ell(\gamma)}\right),
\end{equation}
The product (\ref{prodzeta}) converges absolutely for $ \R s > \delta $.

Studying the resonances of $ X = \Gamma\setminus\HH $ often means studying $ Z_{\Gamma} $. The Selberg zeta function admits a meromorphic continuation to $ s\in \CC $ when $ \Gamma $ is non-elementary and finitely generated, see Selberg \cite{Selbergzeta} and Guillopé \cite{Guil}. We will continue to write $ Z_{\Gamma} $ for the meromorphic continuation of the Selberg zeta function. For non-elementary and geometrically finite hyperbolic surfaces $ X= \Gamma\setminus \HH $ it follows from the results of Selberg \cite{Selbergzeta} (for the finite-area case) and Borthwick-Judge-Perry \cite{BJP} (for the infinite-area case) that the zero set of $ Z_{\Gamma} $ is split into ``topological zeros'' at the non-positive integers $ -\NN_{0} $, and the set $ \mathcal{R}_{X} $ of resonances, counted with multiplicities. (The multiplicity of a resonance $ \zeta $ is given by the rank of the residue operator of $ R_{X}(s) $ at the pole $ s=\zeta $.) In particular, $ Z_{\Gamma} $ has no zeros on the line $ \R s = \delta $ except for a simple zero at $ s= \delta $, and all the other zeros lie in the half-plane $ \R s < \delta. $

In view of these results, Theorem 1 can be phrased in terms of the Selberg zeta function as follows: for every $ \varepsilon > 0 $ and for every $ \Gamma $ satisfying the assumptions of Theorem 1, we can pass to a suitable finite-index normal subgroup $ \Gamma'$ of $ \Gamma $ such that $ Z_{\Gamma'} $ has a zero $ s\neq \delta $ with  $  \vert\delta -s \vert < \varepsilon $.

\paragraph{Acknowledgements} I would like to thank my PhD advisor Anke Pohl for many helpful discussions.

\section{Explicit Construction and Proof Strategy}\label{construction}

Recall that $ \Gamma $ is a non-elementary, finitely generated Fuchsian group without elliptic elements. Then $ X = \Gamma\setminus \HH $ is a two-dimensional smooth Riemannian manifold with constant negative sectional curvature $ -1 $ and geometrically finite. Moreover the Euler characteristic of $ X $ is finite. We refer the reader to \cite[Chapter~2]{Borthwick} for more background material on hyperbolic surfaces.

The fundamental group $ \pi_{1}(X) $ of $ X= \Gamma\setminus \HH $ is isomorphic to $ \Gamma $. Thus the homology group $ H_{1}(X, \ZZ)$ is given by the abelianization $\Gamma/[\Gamma, \Gamma] $ of $ \Gamma $, and it is infinite. We deduce that the homology group is isomorphic to $ \ZZ^{r}\oplus \mathrm{Tor} $, where $ r $ is a positive integer and $ \mathrm{Tor} $ is the (finite) torsion subgroup. (This follows from the structure theorem for finitely generated abelian groups.) Notice that $ r $ is equal to the first Betti number of $ X $.

Choose an isomorphism $ \varphi : H_{1}(X, \ZZ)\to \ZZ^{r}\oplus \mathrm{Tor} $ and let $ T $ be the pre-image of the torsion group $ \mathrm{Tor} $ under $ \varphi $. Clearly, $ T $ is a finite normal subgroup of $ H_{1}(X, \ZZ). $ In this paper we will consider the group $ H:= H_{1}(X, \ZZ)/T $, which is isomorphic to $ \ZZ^{r} $. The group $ H $ may be viewed as the torsion-free part of the homology group.

Let $ [\gamma] $ denote the image of the element $ \gamma\in \Gamma $ under the natural projection $ \Gamma\to H $. We shall assume that an isomorphism $ H\cong \ZZ^{r} $ has been fixed and write $ \ZZ^{r} $ instead of $ H $. We may therefore regard $ [\gamma] $ as a vector in $ \ZZ^{r} $. For $ N\in \mathbb{N} $ we define the subgroups $ \Gamma_{N} $ by 
\begin{equation}\label{Arbi}
\Gamma_{N} := \{ \gamma\in \Gamma : [\gamma] \equiv 0 \mod N \},
\end{equation}
where $ v=(v_{1}, \cdots, v_{r}) \equiv 0 \mod N $ simply means $ v_{j}\equiv 0\mod N $ for all $ j=1,\dots, r. $

As a concrete example, consider the fundamental group $ \Gamma = \pi_{1}(X) $ of a three-funnel surface $ X $. In this case $ \Gamma $ is a free group generated by two isometries $ S_{1}, S_{2} \in \PSL_{2}(\RR) $ and its inverses. By the Nielsen-Schreier theorem \cite{Stillwell}, $ \Gamma_{N} $ is a free, being a subgroup of a free group. For $ N=2 $ one can verify that $ \Gamma_{2} $ is generated by the following elements (and its inverses):
$$ S_{1}^{2},\, S_{2}^{2},\, S_{1}S_{2}^{2}S_{1},\, S_{2}S_{1}^{2}S_{2},\, S_{1}S_{2}S_{1}S_{2}. $$

Let us now recall the ends geometry of non-elementary, geometrically finite hyperbolic surfaces: $ X $ can be decomposed into a convex surface $ \mathcal{N} $ with geodesic boundary, called the \textit{convex core} of $ X $, on which infinite area ends $ F_{i} $ are glued, called \textit{funnels}. A funnel $ F_{i} $ is a half-cylinder isometric to
$$ \left( \RR/\ell_{i}\ZZ\right)_{\theta}\times (\RR^{+})_{t}, $$
endowed with the metric 
$$ ds^{2} = dt^{2} + \cosh^{2}(t) d\theta^{2}. $$
Here $ \ell_{i} $ is the length of the closed boundary geodesic dividing the funnel $ F_{i} $ and Nielsen region $ \mathcal{N} $. The convex core can further be decomposed into a compact set $ \mathcal{K} $, called the \textit{compact core} of $ X $, on which finite-area ends $ C_{j} $ are glued, called \textit{cusps}. Any  cusp $ C_{j} $ is isometric to the quotient
$$ C = \Gamma_{\infty}\setminus \{ z\in \HH : 0\leq \R z \leq 1, \; \I z \geq 1 \}, \quad \Gamma_{\infty} = \langle z\mapsto z+1\rangle, $$ 
and has hyperbolic area equal to $ 1 $. One way to measure how ``big'' $ X $ is, is provided by the hyperbolic area of the convex core $ \mathcal{N} $. For finite-area hyperbolic surfaces, $ \mathrm{area}(\mathcal{N}) $ is equal to the area of $ X $, since there are no funnels in this case. By \cite[Lemma~10.3]{Borthwick} we have 
\begin{equation}\label{char}
\mathrm{area}(\mathcal{N}) = -2\pi \Xi(X),
\end{equation}
where $ \Xi(X) $ is the Euler characteristic of $ X $. 

Returning to the subgroups $ \Gamma_{N} $ defined by (\ref{Arbi}), consider the homomorphism 
$$ \varphi_{N} : \Gamma \to \left(\ZZ/N\ZZ\right)^{r}, \; \gamma \mapsto [\gamma] \mod N. $$ This homomorphism is surjective and thus $ \Gamma_{N} = \ker \varphi_{N} $. Moreover the quotient $ \Gamma/\Gamma_{N} $ is isomorphic to $ \left(\ZZ/N\ZZ\right)^{r} $. Therefore there is a natural covering map
$$ X_{N} = \Gamma_{N}\setminus\HH \to X = \Gamma\setminus\HH  $$
of degree $ [\Gamma : \Gamma_{N}] = N^{r} $. It follows from a standard fact of algebraic topology that $ X_{N} $ has Euler characteristic $ \Xi(X_{N}) = N^{r}\Xi(X) $. We conclude that the area of the convex core $ \mathcal{N}_{N} $ of $ X_{N} $ grows polynomially in $ N $. Indeed, by (\ref{char}) we have $ \mathrm{area}(\mathcal{N}_{N}) = -2\pi N^{r} \Xi(X). $

The surfaces $ X_{N} $ are the covers of $ X $ giving arbitrarily small spectral gap in Theorem 1 and Theorem 2.

\paragraph{Proof Strategy}
In Section \ref{Expanders} we recall the definition of expander graphs and its connection with the spectral gap problem of hyperbolic surfaces with $ \delta > \frac{1}{2} $.  
Exploiting this connection, we reduce Theorem \ref{First Main Theorem} for $ \delta > \frac{1}{2} $ to a simple statement about the spectral gap of the Cayley graph of the group $ \Gamma/\Gamma_{N}\cong (\ZZ/N\ZZ)^{r} $ with respect to a fixed generating set of $ \ZZ^{r} $.

By the work of Patterson \cite{Patt}, we know that hyperbolic surfaces with $ \delta \leq \frac{1}{2} $ cannot have cusp ends, that is, the convex core $ \mathcal{N} $ of $ X $ is compact. In this case $ X $ is called \textit{convex co-compact}. By the result of Button \cite{Button} a non-elementary, geometrically finite, and convex co-compact surface is a Schottky surface. Hence, for $ \delta \leq \frac{1}{2} $ Theorem \ref{First Main Theorem} follows from Theorem \ref{main theorem} and the rest of this paper is devoted to the proof of Theorem \ref{main theorem}.

We show that the Selberg zeta function of $ \Gamma_{N} $ factorizes as a product of certain L-functions $ L_{\Gamma}(s,\chi) $ over all the characters $ \chi $ of the abelian group $ \Gamma/\Gamma_{N} $. This result is a straightforward application of an analogue of the well-known Artin-Takagi formula of algebraic number theory that we prove in Section \ref{Thermodynamic Formalism Approach for L-Functions}.  

For $ N $ large, a large proportion of the characters of $ \Gamma/\Gamma_{N} $ are close to the identity. Theorem \ref{main theorem} is ultimately a consequence of the fact that every L-function $ L_{\Gamma}(s,\chi) $ associated to a Schottky group $ \Gamma $ and a character $ \chi :\Gamma\to S^{1} $ close to the identity $ \chi =1 $ has to vanish at some point $ s $ near $ s=\delta $. Analytic continuation and the vanishing property of such L-functions will follow from the thermodynamic formalism approach, which we consider in Section \ref{L-Functions and Artin-Takagi Formula}.

\section{Spectral Gap and Expanders}\label{Expanders}

\subsection{Proof of Theorem \ref{First Main Theorem} for large Hausdorff dimension}

Let $ \Gamma $ be as in Theorem \ref{First Main Theorem} and let $ \delta $ be the Hausdorff dimension of the limit set $ \Lambda(\Gamma) $. In this section we prove Theorem \ref{First Main Theorem} under the additional assumption $ \delta > \frac{1}{2} $ by exploiting the connection of uniform spectral gap with expander graphs. Let us recall the definition of the latter.

Let $ \mathcal{G} $ be a finite graph with set of vertices $ \mathcal{V} $ and of degree $ k $. That is, for every vertex $ x\in \mathcal{V} $ there are $ k $ edges adjacent to $ x $. For a subset of vertices $ A\subset \mathcal{V} $ we define its boundary $ \partial A $ as the set of edges with one extremity in $ A $ and the other in $ \mathcal{G}-A. $ The Cheeger isoperimetric constant $ h(\mathcal{G}) $ is defined as
$$ h(\mathcal{G}) = \min \left\{  \frac{\vert \partial A \vert}{\vert A\vert} : A\subset \mathcal{V} \text{ and } 1 \leq \vert A\vert \leq \frac{\vert \mathcal{V}\vert}{2}\right\}. $$
Let $ L^{2}(\mathcal{V}) $ be the Hilbert space of complex-valued functions on $ \mathcal{V} $ with inner product
$$ \langle F, G \rangle_{L^{2}(\mathcal{V})} = \sum_{x\in \mathcal{V}} F(x) \overline{G(x)}. $$
Let $ \Delta $ be the discrete Laplace operator acting on $ L^{2}(\mathcal{V}) $ by
$$ \Delta F(x) = F(x) - \frac{1}{k}\sum_{y\sim x} F(y), $$
where $ F\in L^{2}(\mathcal{V}) $, $ x\in \mathcal{V} $ is a vertex of $ \mathcal{G} $, and $ y\sim x $ means that $ y $ and $ x $ are connected by an edge. The operator $ \Delta $ is self-adjoint and positive. For a finite graph, let $ \lambda_{1}(\mathcal{G}) $ denote the first non-zero eigenvalue of $ \Delta. $ 

The following result due to Alon and Milman \cite{AlonMilman} relates the spectral gap $ \lambda_{1}(\mathcal{G}) $ and Cheeger's isoperimetric constant.

\begin{prop}\label{Spectral Gap and Cheeger}
For finite graphs $ \mathcal{G} $ of degree $ k $ we have
\begin{equation}\label{inequality}
\frac{1}{2}k\cdot \lambda_{1}(\mathcal{G}) \leq h(\mathcal{G}) \leq k \sqrt{\lambda_{1}(\mathcal{G})(1-\lambda_{1}(\mathcal{G}))}.
\end{equation}
\end{prop}

We note that large first non-zero eigenvalue $ \lambda_{1}(\mathcal{G}) $ implies fast convergence of random walks on $ \mathcal{G} $, that is, high connectivity (see Lubotzky \cite{Lubotzky2}).

\begin{defi}
A family of finite graphs $ \{ \mathcal{G}_{i} \} $ of a fixed degree is called a family of expanders if there exists a constant $ c>0 $ such that $ h(\mathcal{G}_{i})\geq c. $
\end{defi}

The family of graphs we are interested in are built as follows. Let $ \Gamma=\langle S\rangle $ be a Fuchsian group generated by a finite set $ S\subset \mathrm{PSL}_{2}(\RR) $. We will assume that $ S $ is symmetric, i.e. $ S^{-1}=S. $ Let $ (\Gamma_{n})_{n\in \mathbb{N}} $ a sequence of finite index normal subgroups of $ \Gamma $. For each $ n\in \mathbb{N} $ let $ S_{n} $ denote the image of $ S $ under the natural projection $ \Gamma\to \Gamma/\Gamma_{n}=: G_{n} $. Notice that $ S_{n} $ is a symmetric generating set for the group $ G_{n} $. Let $ \mathcal{G}_{n} = \mathcal{G}(\Gamma/\Gamma_{n}, S_{n}) $ denote the Cayley graph of $ G_{n} $ with respect to the generating set $ S_{n} $. That is, the vertices of $ \mathcal{G}_{n} $ are the elements of $ G_{n} $ and two vertices $ x $ and $ y $ are connected by an edge if and only if $ xy^{-1}\in S_{n}. $

The connection of uniform spectral gap with the graphs constructed above comes from the following result.

\begin{prop}\label{expanders and gap}
Let $ \Gamma $ be a finitely generated Fuchsian group without elliptic elements and $ \delta(\Gamma) > \frac{1}{2} $. Let $ (\Gamma_{n})_{n\in \NN} $ be a sequence of finite-index normal subgroups of $ \Gamma $. Assume there exists $ \varepsilon_{0} > 0 $ such that $ \mathrm{gap}(\Gamma_{n}\setminus \HH) \geq \varepsilon_{0} $ for all $ n\in \NN $. Then the graphs $ \mathcal{G}_{n} = \mathcal{G}(\Gamma/\Gamma_{n}, S_{n}) $ form a family of expanders.
\end{prop}

Let us see how Proposition \ref{expanders and gap} implies Theorem \ref{First Main Theorem} for $ \delta > \frac{1}{2} $.

\begin{proof}[Proof of Theorem \ref{First Main Theorem} for $ \delta > \frac{1}{2} $] 
Let $ \Gamma $ satisfy the assumptions of Theorem \ref{First Main Theorem} together with $ \delta > \frac{1}{2} $ and suppose the theorem is false. Then there exists $ \varepsilon_{0} > 0 $ such that 
$$ \mathrm{gap}'(\Gamma_{N}\setminus \HH) \geq \varepsilon_{0}, $$
where $ \Gamma_{N} = \ker(\varphi_{N}) $ are the subgroups constructed in Section \ref{construction}. Recall that the resonances with $ \R s > \frac{1}{2} $ correspond to $ L^{2} $-eigenvalues $ \lambda = s(1-s) \in [\delta(1-\delta), \frac{1}{4})  $. In particular, all resonances $ s\in \mathcal{R}_{X} $ with $ \R s > \frac{1}{2} $ lie on the real line, so that we may assume 
$$ \mathrm{gap}(\Gamma_{N}\setminus \HH) \geq \varepsilon_{0} $$
for all $ N\in \NN $. Recall that $ \Gamma/ \Gamma_{N} $ is isomorphic to the abelian group $ (\ZZ/N\ZZ)^{r} $. Let $ S $ be a finite, symmetric generating set of $ \Gamma $, let $ S_{N} $ be the image of $ S $ under the projection $ \Gamma \to (\ZZ/N\ZZ)^{r} $, and let $ \mathcal{G}_{N} $ be the Cayley graph of $ (\ZZ/N\ZZ)^{r} $ with respect to $ S_{N} $. 

By Proposition \ref{expanders and gap}, the graphs $ \mathcal{G}_{N} $ form a family of expanders. Thus there exists $ c>0 $ such that for all $ N\in \NN $ we have
\begin{equation}\label{isoperm}
h(\mathcal{G}_{N}) \geq c.
\end{equation}

For $ N $ large enough the projection $ S \to (\ZZ/N\ZZ)^{r} $ is injective, because $ S $ is a finite set. Since we are only interested in large $ N $, we may therefore assume that the $ \mathcal{G}_{N} $ are Cayley graphs of $ (\ZZ/N\ZZ)^{r} $ with respect to a fixed set of $ S\subset \ZZ^{r} $.

The space $ L^{2}\left( (\ZZ/N\ZZ)^{r} \right) $ is spanned by the functions $ f_{a}(x)=e\left( \frac{1}{N}\langle a, x\rangle \right) $ for $ a\in (\ZZ/N\ZZ)^{r} $. Here, $ e(t):=e^{2\pi it} $ and $ \langle a, x\rangle := a_{1}x_{1} + \cdots + a_{r}x_{r}. $
Moreover,
\begin{align*}
\Delta f_{a}(x) &=  f_{a}(x) - \frac{1}{\vert S\vert}\sum_{s\in S} f_{a}(x+s)\\
&= f_{a}(x) - \frac{1}{\vert S\vert}\sum_{s\in S} e\left( \frac{1}{N}\langle a, s\rangle\right) f_{a}(x)\\
&= f_{a}(x) - \frac{1}{\vert S\vert} \sum_{s\in S} \cos\left( \frac{2\pi}{N} \langle a, s\rangle \right) f_{a}(x)\\
&= \frac{1}{\vert S\vert} \sum_{s\in S} \left( 1- \cos\left( \frac{2\pi}{N} \langle a, s\rangle \right)\right) f_{a}(x),
\end{align*}
where we exploited the symmetry of $ S $ in the third line. Thus for every $ a\in (\ZZ/N\ZZ)^{r}  $ the function $ f_{a} $ is an eigenfunction of $ \Delta $ with eigenvalue
$$ \lambda_{a} = \frac{1}{\vert S\vert} \sum_{s\in S} \left( 1- \cos\left( \frac{2\pi}{N} \langle a, s\rangle \right)\right). $$
Fix an element $ a \in \ZZ^{r} $ which is not orthogonal to all the elements in $ S $. For $ N $ large enough, we can identify $ a \mod N \in (\ZZ/N\ZZ)^{r} $ with $ a $ and we will continue to write $ a $. Then we have $ \lambda_{a} > 0 $. 
Using $ 1-\cos(2\pi x) \ll x^{2} $ for $ \vert x\vert $ sufficiently small, we deduce that
$$ 0<\lambda_{a}\ll \frac{1}{N^{2}}.  $$
It follows that the first non-zero eigenvalue $ \lambda_{1}(\mathcal{G}_{N}) $ is bounded by $ O(\frac{1}{N^{2}}) $. With the second inequality in (\ref{inequality}) we conclude that $ h(\mathcal{G}_{N}) = O(\frac{1}{N}), $ which contradicts (\ref{isoperm}).
\end{proof}

\subsection{Proof of Proposition \ref{expanders and gap}}

A very similar statement to that of Proposition \ref{expanders and gap} was given by Gamburd \cite[Section~7]{Gamburd}. The key ingredient in Gamburd's proof is Fell's continuity of induction and we will follow this line of thought. 

For the remainder of this section set $ G = \SL_{2}(\RR) $ and let $ \hat{G} $ be its unitary dual, that is, the set of equivalence classes of (continuous) irreducible unitary representations of $ G $. We endow the set $ \hat{G} $ with the \textit{Fell topology}. We refer the reader to \cite{fell} and  \cite[Chapter~F]{Property(T)} for more background on the Fell topology. A representation of $ G $ is called \textit{spherical} if it has a non-zero $ K $-invariant vector, where $ K = \mathrm{SO}(2) $. Let us consider the subset $ \hat{G}^{1} \subset \hat{G} $ of irreducible spherical unitary representations. 

According to Lubotzky \cite[Chapter~5]{Lubotzky}, the set $ \hat{G}^{1} $ can be parametrized as 
$$ \hat{G}^{1} = i\RR^{+} \cup \left[ 0, \frac{1}{2}\right], $$
where $ s\in i\RR^{+} $ corresponds to the spherical unitary principal series representations, $ s\in (0, \frac{1}{2}) $ corresponds to the complementary series representation, and $ s = \frac{1}{2} $ corresponds to the trivial representation. See also Gelfand, Graev, Pyatetskii-Shapiro \cite[Chapter~1 §3]{Gelfand} for a classification of the irreducible (spherical and non-spherical) unitary representations with a different parametrization. Moreover the Fell topology on $ \hat{G}^{1} $ is the same as that induced by viewing the set of parameters $ s $ as a subset of $ \CC $, see \cite[Chapter~5]{Lubotzky}. In particular, the spherical unitary principal series representations are bounded away from the identity.

Let us now recall the connection between the exceptional eigenvalues $ \lambda \in (0, \frac{1}{4}) $ and the complementary series representation.  Consider the (left) quasiregular representation $ (\lambda_{G/\Gamma}, L^{2}(G/\Gamma)) $ of $ G $ defined by 
$$ \lambda_{G/\Gamma}(g)f(h\Gamma) = f(hg^{-1}\Gamma). $$ 
(We will denote this representation simply by $ L^{2}(G/\Gamma) $.) Define the function $ s(\lambda) = \sqrt{1/4 - \lambda} $ for $ \lambda\in (0, \frac{1}{4}) $. Then, $ \lambda\in (0, \frac{1}{4}) $ is an exceptional eigenvalue of $ \Delta_{\Gamma\setminus \HH} $ if and only if the complementary series $ \pi_{s(\lambda)} $ occurs as a subrepresentation of $  L^{2}(G/\Gamma) $. This is the so-called Duality Theorem \cite[Chapter~1 §4]{Gelfand}.

Let us return to the proof of Proposition \ref{expanders and gap}. Let $ \Gamma $ and $ (\Gamma_{n})_{n\in \NN} $ be as in Proposition \ref{expanders and gap}. Let $ \lambda_{0}(\Gamma) = \delta(1-\delta) = \inf \Omega(\Gamma\setminus \HH) $ denote the bottom of the spectrum of $ \Delta_{\Gamma\setminus \HH} $ on $ L^{2}(\Gamma\setminus \HH) $. Since $ \Gamma_{n} $ is a finite-index subgroup of $ \Gamma $, we have $ \delta(\Gamma_{n}) = \delta $ and consequently
$$ \lambda_{0}(\Gamma_{n}) = \lambda_{0}(\Gamma) =:\lambda_{0} $$
for each $ n\in \NN $. 
Let $ V_{s_{0}} $ be the invariant subspace corresponding to the representation $ \pi_{s_{0}} $ and let $ L_{0}^{2}(G/\Gamma_{n}) $ be its orthogonal complement in $ L^{2}(G/\Gamma_{n}). $ For each $ n\in \NN $ we can decompose the quasiregular representation of $ G $ into direct sum of subrepresentations 
$$ L^{2}(G/\Gamma_{n}) =  L_{0}^{2}(G/\Gamma_{n}) \oplus V_{s_{0}}. $$
Recall that $ \lambda_{0} $ is a simple eigenvalue by the result of Patterson \cite{Patt2}. By the Duality Theorem it follows that $ V_{s_{0}} $ is one-dimensional. The following lemma provides us with a link between uniform spectral gap and representation theory.

\begin{lem}[Representation-theoretic reformulation of uniform spectral gap]\label{Intermediate}
Let $ \mathcal{R} \subset \hat{G}^{1} $ be the following set:
$$ \mathcal{R} = \bigcup_{n\in \mathbb{N}} \{ (\pi, \mathcal{H}) : \pi \text{ is spherical irreducible unitary subrep. of } L^{2}_{0}(G/\Gamma_{n}) \} / \sim, $$
where $ \sim $ denotes the equivalence of representations. Then the following are equivalent.
\begin{itemize}
\item[(i)] There exists $ \varepsilon_{0} > 0 $ such that $ \mathrm{gap}(\Gamma_{n}\setminus \HH) \geq \varepsilon_{0} $ for all $ n\in \NN $.
\item[(ii)] The representation $ \pi_{s_{0}} $ is isolated in the set $ \mathcal{R} \cup \{ \pi_{s_{0}}\} $ with respect to the Fell topology. 
\end{itemize} 
\end{lem}

\begin{proof}
Since the resonances $ s\in \mathcal{R}_{X_{n}} $ with $ \R s > \frac{1}{2} $ correspond to the eigenvalues $ \lambda = s(1-s)\in [\lambda_{0}, \frac{1}{4}) $, the uniform spectral gap condition (i) can be stated as follows. There exists $ \varepsilon_{1} > 0 $ such that for all $ n\in \NN $ we have
\begin{equation}\label{eigenvalue gap}
\Omega(\Gamma_{n}\setminus \HH) \cap [0, \lambda_{0} + \varepsilon_{1}) = \{ \lambda_{0} \}.
\end{equation} 
Now we can reformulate (\ref{eigenvalue gap}) in representation-theoretic language. Set $ s_{0} = s(\lambda_{0}). $ Then by the Duality Theorem, there exists $ \varepsilon > 0 $ such that for all $ n\in \NN $ and all $ s\in (s_{0}-\varepsilon, \frac{1}{2}] $, the complementary series representation $ \pi_{s} $ does not occur as a subrepresentation of $ L^{2}(G/\Gamma_{n}). $ Since $ V_{s_{0}} $ is one-dimensional (and each representation $ \pi_{s} $ with $ s\neq \frac{1}{2} $ is infinite-dimensional), (i) is equivalent to 
\begin{equation}\label{isolation}
\mathcal{R} \cap \left( s_{0}-\varepsilon, \frac{1}{2}\right] = \{  s_{0}\}. 
\end{equation}
Since the Fell topology on $ \hat{G}^{1} $ is equivalent to the one induced by viewing $ \hat{G}^{1} $ as the subset $ i\RR^{+} \cup \left[ 0, \frac{1}{2}\right] $ of the the complex plane, the equivalence of (i) and (ii) is now evident. 
\end{proof}

Let $ L^{2}(\Gamma/\Gamma_{n}) $ be the Hilbert space of complex valued functions on the (finite) set of left cosets $ \Gamma/\Gamma_{n} $ with inner product
$$ \langle F,G \rangle_{L^{2}(\Gamma/\Gamma_{n})} = \sum_{x\in \Gamma/\Gamma_{n}} F(x) \overline{G(x)}, $$
and let $ L_{0}^{2}(\Gamma/\Gamma_{n}) = \{ 1 \}^{\perp} $ be the subspace of functions orthogonal to the constant functions. Let $ 1_{\Gamma_{n}} $ denote the trivial representation of $ \Gamma_{n} $ on $ \CC $. Then the induced representation $ Ind_{\Gamma_{n}}^{\Gamma} 1_{\Gamma_{n}} $ is equivalent to the (left) quasiregular representation $ (\lambda_{\Gamma/\Gamma_{n}}, L^{2}(\Gamma/\Gamma_{n})) $ of $ \Gamma $  defined by 
$$ (\lambda_{\Gamma/\Gamma_{n}}(\gamma)F)(h\Gamma_{n} ) = (\gamma.F)(h\Gamma_{n})  = F(h\gamma^{-1} \Gamma_{n}). $$
The action of $ \Gamma $ on $ L^{2}(\Gamma/\Gamma_{n}) $ given by $ \gamma.F = \lambda_{\Gamma/\Gamma_{n}}(\gamma)F $ is transitive. Hence the only $ \Gamma $-fixed vectors are the constants. Thus we can decompose the representation of $ \Gamma $ on $ L^{2}(\Gamma/\Gamma_{n}) $ into a direct of subrepresentations 
$$ L^{2}(\Gamma/\Gamma_{n}) = L_{0}^{2}(\Gamma/\Gamma_{n}) \oplus \CC, $$
and $ (1_{\Gamma}, \CC) $ does not occur as a subrepresentation of $ L_{0}^{2}(\Gamma_{n}\setminus \HH). $

Consider the following subset of $ \hat{\Gamma} $:
$$ \mathcal{T} = \bigcup_{n\in \mathbb{N}} \{ (\rho, V) : \rho \text{ is irreducible unitary subrepresentation of } L^{2}_{0}(\Gamma/\Gamma_{n}) \} / \sim, $$
We claim the following.

\begin{lem}\label{before expanders}
Assume that one of the equivalent statements in Lemma \ref{Intermediate} holds true. Then the trivial representation $ 1_{\Gamma} $ is isolated in $ \mathcal{T} \cup \{ 1_{\Gamma} \} $ with respect to the Fell topology.
\end{lem}

\begin{proof} Let $ K $ be a closed subgroup of a locally compact group $ H $. Given a unitary representation $ (\pi, V) $ of $ K $, the induced representation $ Ind_{K}^{H} \pi $ of $ H $ is defined as follows. Let $ \mu $ be a quasi-invariant regular Borel measure on $ H/K $ and set
\begin{equation}\label{induced representation}
Ind_{K}^{H} \pi := \{ f : H\to V : f(hk) = \pi(k^{-1})f(h) \text{ for all } k\in K \text{ and } f\in L^{2}_{\mu}(H/K) \}.
\end{equation}
Note that the requirement $ f\in L^{2}_{\mu}(H/K) $ makes sense, since the norm of $ f(g) $ is constant on each left coset of $ H $. The action of $ G $ on $ Ind_{H}^{G} \pi $ is defined by
$$ g.f(x) = f(g^{-1}x) $$ 
for all $ x,g\in G $, $ f\in Ind_{H}^{G} \pi. $ We also note that the equivalence class of the induced representation $ Ind_{K}^{H} \pi $ is independent of the choice of $ \mu $. We refer the reader to \cite[Chapter~E]{Property(T)} for a more thorough discussion on properties of induced representations.

If two representations $ (\pi_{1}, \mathcal{H}_{1}) $ and $ (\pi_{2}, \mathcal{H}_{2}) $ are equivalent, we write $ \mathcal{H}_{1} = \mathcal{H}_{2} $ by abuse of notation. Using induction by stages (see \cite{Folla-95} or \cite{Gaal-73} for a proof) we have
\begin{align*}
V_{s_{0}} \oplus L_{0}^{2}(G/\Gamma_{n}) &= L(G/\Gamma_{n})\\
&= Ind_{\Gamma_{n}}^{G} 1_{\Gamma_{n}}\\
&= Ind_{\Gamma}^{G} Ind_{\Gamma_{n}}^{\Gamma} 1_{\Gamma_{n}}\\
&= Ind_{\Gamma}^{G} L^{2}(\Gamma/\Gamma_{n})\\
&= Ind_{\Gamma}^{G} 1_{\Gamma} \oplus Ind_{\Gamma}^{G} L_{0}^{2}(\Gamma/\Gamma_{n})\\
&= V_{s_{0}} \oplus L_{0}^{2}(G/\Gamma) \oplus Ind_{\Gamma}^{G} L_{0}^{2}(\Gamma/\Gamma_{n}).
\end{align*}
Choose an index $ n\in \NN $ and an irreducible unitary subrepresentation $ (\tau, V) $ of $ L_{0}^{2}(\Gamma/\Gamma_{n}) $. The above calculation implies that $ Ind_{\Gamma}^{G} \tau $ is a unitary subrepresentation of $ L_{0}^{2}(G/\Gamma_{n}) $. Since $ \tau $ is unitary and irreducible, so is $ Ind_{\Gamma}^{G} \tau $. Moreover $ Ind_{\Gamma}^{G} \tau $ is a spherical representation of $ G $, since any non-zero function $ f\in L^{2}(\HH/\Gamma) $ and non-zero vector $ v\in V $ gives rise to a non-zero $ K $-invariant function $ F\in Ind_{\Gamma}^{G} \tau $. Indeed, we have $ \HH \cong K\setminus G $, so that we my view $ f $ as function $ f : G\to \CC $ satisfying $ f(kg\gamma) = f(g) $ for all $ g\in G, k\in K, \gamma\in \Gamma $. Now one easily verifies that $ F = fv : G\to V $ belongs to $ Ind_{\Gamma}^{G} \tau $ and is invariant under $ K $. In other words, $ Ind_{\Gamma}^{G} \tau $ belongs to $ \mathcal{R} $.

Now suppose the lemma is false. Then there exists a sequence $ (\tau_{j})_{j\in \NN} \subset \mathcal{T} $ that converges to $ 1_{\Gamma} $ as $ j\to \infty $. On the other hand, $ \pi_{s_{0}} $ is weakly contained in $ Ind_{\Gamma}^{G} 1_{\Gamma} $. By Fell's continuity of induction \cite{fell} we have 
$$ \pi_{s_{0}} \prec Ind_{\Gamma}^{G} 1_{\Gamma} = \lim_{j\to \infty} Ind_{\Gamma}^{G} \tau_{j} \in \overline{\mathcal{R}},   $$
which contradicts Lemma \ref{Intermediate}.
\end{proof}

We can now prove Proposition \ref{expanders and gap}.

\begin{proof}[Proof of Proposition \ref{expanders and gap}]
Let us recall the definition of the Fell topology on $ \hat{\Gamma} $ (for further reading consult \cite[Chapter~F]{Property(T)}). For an irreducible unitary representation $ (\pi, V) $ of $ \Gamma $, for a unit vector $ \xi\in V $, for a finite set $ Q\subset \Gamma $, and for $ \varepsilon > 0 $ let us define the set  $ W(\pi, \xi, Q, \varepsilon) $ that consists of all irreducible unitary representations $ (\pi', V') $ of $ \Gamma $ with the following property. There exists a unit vector $ \xi'\in V' $ such that
$$ \sup_{\gamma\in Q} \vert \langle \pi(\gamma)\xi, \xi\rangle_{V} - \langle \pi'(\gamma)\xi', \xi'\rangle_{V'} \vert < \varepsilon. $$
The Fell topology is generated by the sets $ W(\pi, \xi, Q, \varepsilon) $. By Lemma \ref{before expanders} and the definition of the Fell topology, there exists $ c_{0} = c_{0}(\Gamma, S) > 0 $ only depending on $ \Gamma $ and the generating set $ S $ of $ \Gamma $, but not on $ n $, such that for all $ F\in L_{0}^{2}(\Gamma/\Gamma_{n}) $
\begin{equation}\label{Implication}
\sup_{\gamma\in S} \vert \langle \gamma.F - F, F \rangle_{L^{2}(\Gamma/\Gamma_{n})}  \vert  \geq c_{0}\Vert F\Vert^{2}.
\end{equation}
By the Cauchy-Schwarz inequality we have 
$$ \sup_{\gamma\in S}\Vert \gamma .F - F \Vert  \geq c_{0}\Vert F\Vert. $$
Set $ \mathcal{V}_{n} = \Gamma/\Gamma_{n} $. Fix a non-empty subset $ A $ of $ \mathcal{V}_{n} $ with $ \vert A\vert \leq \frac{1}{2}\vert\mathcal{V}_{n}\vert $ and define the function
$$ F(x) = 
\begin{cases}
\vert \mathcal{V}_{n}\vert - \vert A\vert, & \text{ if } x\in A\\
-\vert A\vert & \text{ if } x\notin A.
\end{cases}
$$
One can verify that $ F\in L_{0}^{2}(\Gamma/\Gamma_{n}) $ and $ \Vert F\Vert^{2} = \vert A\vert\vert \mathcal{V}_{N} \vert (\vert \mathcal{V}_{N}\vert - \vert A\vert). $ On the other hand,
$$ \Vert \gamma .F - F \Vert^{2} = \vert \mathcal{V}_{n}\vert^{2} E_{\gamma}(A, \mathcal{V}_{n}\setminus A), $$
where
$$ E_{\gamma}(A,B) := \left\vert \{  x\in \mathcal{V}_{N} :  x\in A \text{ and } x\gamma \in B \text{ or } x\in B \text{ and } x\gamma \in A \} \right\vert . $$
Therefore there exists $ \gamma\in S $ such that
\begin{align*}
E_{\gamma}(A, \mathcal{V}_{N}\setminus A)
= \frac{\Vert \gamma. F - F \Vert^{2}}{\vert \mathcal{V}_{N}\vert^{2}}
\geq \frac{c_{0}^{2}\Vert F\Vert^{2}}{\vert \mathcal{V}_{N}\vert^{2}} = c_{0}^{2} \left( 1-\frac{\vert A\vert}{\vert \mathcal{V}_{N}\vert} \right) \vert A\vert.
\end{align*}
Thus we obtain a lower bound for the size of the boundary of the $ A $ in the graph $ \mathcal{G}_{n} $:
$$ \vert \partial A\vert \geq \frac{1}{2} \sup_{\gamma\in S} E_{\gamma}(A, \mathcal{V}_{n}\setminus A) \geq  \frac{c_{0}^{2}}{2} \left( 1-\frac{\vert A\vert}{\vert \mathcal{V}_{n}\vert} \right) \vert A\vert \geq \frac{c_{0}^{2}}{4} \vert A\vert. $$
Consequently, $ h(\mathcal{G}_{n}) \geq c_{0}^{2}/4 $ for all $ n\in \NN. $ In other words, the graphs $ \mathcal{G}_{n}=\mathcal{G}(\Gamma/\Gamma_{n}, S) $ form a family of expanders.
\end{proof}

\section{L-Functions and Artin-Takagi Formula}\label{L-Functions and Artin-Takagi Formula}

In this section we turn our attention to \textit{L-functions}.

\begin{defi}
For a hyperbolic surface $ X=\Gamma\setminus \HH $ and a finite-dimensional unitary representation $ \rho : \Gamma \to \mathrm{U}(V) $, the L-function $ L_{\Gamma}(s,\rho) $ is defined for $ \R s \gg 0 $ by the formal product
\begin{equation}\label{prod}
L_{\Gamma}(s,\rho) := \prod_{\overline{\gamma}\in \overline{\Gamma}_{p}} \prod_{k=0}^{\infty} \det\left( 1-\rho(\gamma)e^{-(s+k)\ell(\gamma)}\right).
\end{equation}
\end{defi}

Assume that $ (\rho, V) $ is a unitary representation of dimension $ d $. Then for all $ \gamma\in \Gamma $ the eigenvalues of $ \rho(\gamma)\in \mathrm{End}(V) $ all lie on the unit circle $ \{ z\in \mathbb{C} : \vert z\vert = 1 \}. $ Therefore 
\begin{equation}\label{estimate}
\left( 1-e^{-(\R s+k)\ell(\gamma)}\right)^{d} \leq \left\vert \det\left( 1-\rho(\gamma)e^{-(s+k)\ell(\gamma)}\right)  \right\vert \leq \left( 1+e^{-(\R s+k)\ell(\gamma)}\right)^{d} .
\end{equation}
Combining (\ref{estimate}) with the product expression (\ref{prod}), we obtain the lower bound $ Z_{\Gamma}(\R s)^{d} \leq \vert  L_{\Gamma}(s,\rho)\vert $. Using the elementary estimate $ 1+x\leq e^{x} $, we obtain the upper bound
\begin{align*}
\vert  L_{\Gamma}(s,\rho)\vert &\leq \prod_{\overline{\gamma}\in \overline{\Gamma}_{p}} \prod_{k=0}^{\infty} \left( 1+e^{-(\R s+k)\ell(\gamma)}\right)^{d}\\
&\leq \exp\left( d \sum_{\overline{\gamma}\in \overline{\Gamma}_{p}} \sum_{k=0}^{\infty} e^{-(\R s+k)\ell(\gamma)} \right)\\
&= \exp\left( d \sum_{\overline{\gamma}\in \overline{\Gamma}_{p}} e^{-\ell(\gamma) \R s}\frac{1}{1-e^{-\ell(\gamma)}} \right).
\end{align*}
Letting $ \ell_{0}>0 $ denote the minimal length among all the closed geodesics of $ \Gamma\setminus \HH $, we deduce that
\begin{equation}\label{trivialineqL}
Z_{\Gamma}(\R s)^{d} \leq \vert  L_{\Gamma}(s,\rho)\vert \leq \exp\left( \frac{d}{1-e^{-\ell_{0}}} \sum_{\overline{\gamma}\in \overline{\Gamma}_{p}} e^{-\ell(\gamma) \R s} \right).
\end{equation}
Since the series 
$$ \sum_{\overline{\gamma}\in \overline{\Gamma}_{p}} e^{-\ell(\gamma) s} $$ 
converges absolutely and uniformly on compact sets of the half-plane $ \R s > \delta $, so does the infinite product (\ref{prod}). Moreover the lower bound in (\ref{trivialineqL}) implies that $ L_{\Gamma}(s, \rho) $ has no zeros in $ \R s > \delta $. In Section \ref{Thermodynamic Formalism Approach for L-Functions} we will see that $ L_{\Gamma}(s,\rho) $ admits an analytic continuation to $ s\in \CC $ when $ \Gamma $ is a Schottky group.

The L-functions we are considering here should be regarded as analogues of Artin L-functions from algebraic number theory. Given an arbitrary number field $ k $ and a relative Galois extension $ K $, there is a formula relating the zeta function of $ k $ with that of $ K $ via the Artin L-Functions, proven for abelian Galois groups by Takagi \cite{Takagi} and for general Galois groups by Artin \cite{Artin}. 

For a normal subgroup $ \Gamma'\leqslant \Gamma $ of finite index, we may view the corresponding   quotient $ X'=\Gamma'\setminus \HH $ as a ``Galois'' covering of $ X=\Gamma\setminus\HH $ of finite degree. Let $ \widehat{\Gamma/\Gamma'} $ denote the set of all finite-dimensional irreducible unitary representations of $ \Gamma/\Gamma' $. Obviously, every representation of $ \Gamma/\Gamma' $ can be regarded as a representation of $ \Gamma $ whose kernel contains $ \Gamma' $. Thus it is possible to define $ L_{\Gamma}(s, \rho) $ for every $ \rho\in \widehat{\Gamma/\Gamma'}. $ Note that for the trivial one-dimensional representation we have  $ L_{\Gamma}(s,1) = Z_{\Gamma}(s). $

Using spectral methods, Venkov \cite{Venkov} showed that there is an analogue of the Artin-Takagi formula for the Selberg zeta function $ Z_{\Gamma}(s) $ associated to lattices $ \Gamma\leqslant \mathrm{PSL}_{2}(\RR) $. We shall now give a proof for arbitrary finitely generated Fuchsian groups.

\begin{prop}[Analogue of the Artin-Takagi Formula]\label{Analogue of the Artin-Takagi Formula} Let $ \Gamma $ be a finitely generated discrete group of $ \mathrm{PSL}_{2}(\mathbb{R}) $ and let $ \Gamma' $ be a normal subgroup of $ \Gamma $ of finite index. Then for $ \R s > \delta $ we have 
$$ Z_{\Gamma'}(s) = \prod_{\rho\in \widehat{\Gamma/\Gamma'}} L_{\Gamma}(s,\rho)^{\dim \rho}. $$
\end{prop}

\begin{proof}
Set $ G = \Gamma/\Gamma' $. For $ \R s > \delta $, taking the logarithm of the product definition (\ref{prod}) gives 
\begin{align*}
\log L_{\Gamma}(s,\rho) &= \sum_{\overline{\gamma}\in \overline{\Gamma}_{p}} \sum_{k=0}^{\infty} \log \det\left( 1-\rho(\gamma)e^{-(s+k)\ell(\gamma)}\right)\\
 &= \sum_{\overline{\gamma}\in \overline{\Gamma}_{p}} \sum_{k=0}^{\infty} \tr \log \left( 1-\rho(\gamma)e^{-(s+k)\ell(\gamma)}\right).
\end{align*}
Expanding the logarithm, we obtain
$$ \log L_{\Gamma}(s,\rho) = -\sum_{\overline{\gamma}\in \overline{\Gamma}_{p}} \sum_{k=0}^{\infty} \sum_{m=1}^{\infty} \frac{1}{m}\chi_{\rho}(\gamma^{m})e^{-(s+k)m\ell(\gamma)}, $$
where $ \chi_{\rho}(\gamma) := \tr\, \rho(\gamma) $ denotes the character associated to the representation $ \rho\in \widehat{G} $ (extended to $ \Gamma $ in the obvious way). 
Thus for $ \R s > \delta $ the logarithmic derivative of $ L_{\Gamma}(s,\rho) $ can be expressed as
\begin{align*}
\frac{L_{\Gamma}'(s,\rho)}{L_{\Gamma}(s,\rho)} &= - \frac{d}{ds} \sum_{\overline{\gamma}\in \overline{\Gamma}_{p}} \sum_{k=0}^{\infty} \sum_{m=1}^{\infty} \frac{1}{m}\chi_{\rho}(\gamma^{m})e^{-(s+k)m\ell(\gamma)}\\
&= \sum_{\overline{\gamma}\in \overline{\Gamma}_{p}} \sum_{k=0}^{\infty}  \sum_{m=1}^{\infty} \ell(\gamma)\chi_{\rho}(\gamma^{m})e^{-(s+k)m\ell(\gamma)}.
\end{align*}
Using the geometric series expansion we can rewrite this as
\begin{equation}\label{logder}
\frac{L_{\Gamma}'(s,\rho)}{L_{\Gamma}(s,\rho)}=\sum_{\overline{\gamma}\in \overline{\Gamma}_{p}}\sum_{m=1}^{\infty} \ell(\gamma) \chi_{\rho}(\gamma^{m}) \frac{e^{-sm\ell(\gamma)}}{1-e^{-m\ell(\gamma)}}.
\end{equation}
Let $ e $ be the identity in $ G $. From the representation theory of finite groups there is the well-known identity
$$ \sum_{\rho\in \widehat{G}} \dim \rho\cdot \chi_{\rho}(g) = \vert G\vert \cdot \mathbbm{1}_{e}(g),
$$
where $ \mathbbm{1}_{e} $ is the indicator function of the set $ \{ e\} $. This identity can be extended for all $ \gamma\in \Gamma $ as follows:
\begin{equation}\label{Identity}
\sum_{\rho\in \widehat{G}} \dim \rho\cdot \chi_{\rho}(\gamma) = \vert G\vert \cdot \mathbbm{1}_{\Gamma'}(\gamma).
\end{equation}
Here $ \mathbbm{1}_{\Gamma'} $ is the indicator function of $ \Gamma' $. Let $ \gamma\in \Gamma $ be a non-trivial element and let $ n(\gamma) $ denote the smallest positive integer $ n $ such that $ \gamma^{n}\in \Gamma' $. If there is no such integer set $ n(\gamma) = \infty. $ Let $ \gamma\in \Gamma $ with $ n(\gamma)<\infty $. Then $ \gamma^{n(\gamma)} $ is a $ \Gamma' $-primitive element of $ \Gamma' $. 

Choose representatives of $ G $ in $ \Gamma $. More precisely, set $ n = \vert G\vert $ and fix elements $ g_{1},\dots, g_{n}\in \Gamma $ such that $ \pi(g_{i})\neq \pi(g_{j}) $ for all $ i\neq j, $ where $ \pi:\Gamma\to \Gamma/\Gamma' $ is the natural projection.

Observe that the conjugacy class $ [\gamma^{n(\gamma)}] $ in $ \Gamma $ splits into exactly $ n $ conjugacy classes in $ \Gamma' $. Indeed, the elements $ g_{j}\gamma^{n(\gamma)}g_{j}^{-1}$, $ j=1,\dots, n, $ are mutually conjugate in $ \Gamma $, but they all belong to different $ \Gamma' $-conjugacy classes.

Let $ \{ \gamma_{k} : k\in \mathbb{N}\} $ be a complete list of representatives of $ \Gamma $-conjugacy classes of non-trivial $ \Gamma $-primitive hyperbolic elements with $ n(\gamma) < \infty $. Then
\begin{equation}\label{list}
\bigcup_{j=1}^{n} \{ g_{j}\gamma_{k}^{n(\gamma_{k})}g_{j}^{-1} : k\in \mathbb{N}\}
\end{equation}
provides a complete list of representatives of non-trivial $ \Gamma' $-conjugacy classes of $ \Gamma' $-primitive hyperbolic elements.
Notice that $ \ell(g\gamma g^{-1}) = \ell(\gamma) $ for all $ g,\gamma\in \Gamma $. Therefore summing up (\ref{logder}) over all $ \rho\in \widehat{G} $ and using identity (\ref{Identity}) yields
\begin{align*}
\sum_{\rho\in \widehat{G}} \dim \rho \cdot \frac{L_{\Gamma}'(s,\rho)}{L_{\Gamma}(s,\rho)} &= \vert G\vert \sum_{k=1}^{\infty}\sum_{m=1}^{\infty} \ell(\gamma_{k}) \mathbbm{1}_{\Gamma'}(\gamma_{k}^{m}) \frac{e^{-sm\ell(\gamma_{k})}}{1-e^{-m\ell(\gamma_{k})}}\\ 
 &= \sum_{j=1}^{n} \sum_{k=1}^{\infty}\sum_{m=1}^{\infty} \ell(g_{j}\gamma_{k}g_{j}^{-1}) \mathbbm{1}_{\Gamma'}(g_{j}\gamma_{k}g_{j}^{-1}) \frac{e^{-sm\ell(g_{j}\gamma_{k}g_{j}^{-1})}}{1-e^{-m\ell(g_{j}\gamma_{k}g_{j}^{-1})}}\\
 &= \sum_{j=1}^{n} \sum_{k=1}^{\infty}\sum_{m=1}^{\infty} \ell(g_{j}\gamma_{k}^{n(\gamma_{k})}g_{j}^{-1})  \frac{e^{-sm\ell(g_{j}\gamma_{k}^{n(\gamma_{k})}g_{j}^{-1})}}{1-e^{-m\ell(g_{j}\gamma_{k}^{n(\gamma_{k})}g_{j}^{-1})}} 
\end{align*}
To complete the proof we need to show that the elements in the list (\ref{list}) only appear once. Suppose  $ g_{j}\gamma_{k}^{n(\gamma_{k})}g_{j}^{-1} = g_{i}\gamma_{l}^{n(\gamma_{l})}g_{i}^{-1} $ for some indices $ k,l\in \mathbb{N} $ and $ 1\leq i,j\leq n $. Consequently $ (g_{j}\gamma_{k}g_{j}^{-1})^{n(\gamma_{k})} = (g_{i}\gamma_{l}g_{i}^{-1})^{n(\gamma_{l})} $. Since both $ g_{j}\gamma_{k}g_{j}^{-1} $ and $ g_{i}\gamma_{l}g_{i}^{-1} $ are $ \Gamma $-primitive elements, we deduce that $ n(\gamma_{k}) = n(\gamma_{l}) $ and $ g_{j}\gamma_{k}g_{j}^{-1} = g_{i}\gamma_{l}g_{i}^{-1} $. Hence $ \gamma_{k} $ and $ \gamma_{l} $ are primitive elements conjugate to each other in $ \Gamma $. This forces $ k=l $ and $ i=j $. We conclude that
$$ \sum_{\rho\in \widehat{G}} \dim \rho \cdot \frac{L_{\Gamma}'(s,\rho)}{L_{\Gamma}(s,\rho)} =  \sum_{\overline{\gamma}\in \overline{\Gamma}_{p}}\sum_{m=1}^{\infty} \ell(\gamma)  \frac{e^{-sm\ell(\gamma)}}{1-e^{-m\ell(\gamma)}} = \frac{Z_{\Gamma'}'(s)}{Z_{\Gamma'}(s)}. $$
Since $ Z_{\Gamma'}(s) $ and $  \prod_{\rho\in \widehat{G}} L_{\Gamma}(s,\rho)^{\dim \rho} $ are analytic functions on $ \R s > \delta $ and their logarithmic derivatives coincide, it follows that
$$ Z_{\Gamma'}(s) = C\prod_{\rho\in \widehat{G}} L_{\Gamma}(s,\rho)^{\dim \rho} $$
for some constant $ C $. To see that $ C=1 $, notice that  $ Z_{\Gamma'}(s) $ and $ L_{\Gamma}(s, \rho) $ tend to $ 1 $ as $ \R s\to +\infty $ by the product expressions (\ref{prodzeta}) and (\ref{prod}).
\end{proof}

\section{Thermodynamic Formalism for L-Functions} \label{Thermodynamic Formalism Approach for L-Functions}

\subsection{Fredholm Determinant Identity}

The ``thermodynamic formalism'' for zeta functions introduced by Ruelle \cite{Ruelle} allows us to express the Selberg zeta function as the Fredholm determinant of a certain transfer operator. It turns out that the L-functions introduced in Section \ref{L-Functions and Artin-Takagi Formula} may be expressed as Fredholm determinants of ``twisted'' transfer operators. The thermodynamic formalism approach for L-functions is crucial to our proof of Theorem \ref{main theorem}. We begin this section by recalling the definition of a Fuchsian Schottky group.

Fix an integer $ r\geq 1 $. Let $ D_{1}, \dots, D_{2r} $ be open euclidean disks in $ \CC $ with pairwise disjoint closures and centers on the real axis. The disks can be arranged in any order on the real axis. For each $ j\in \lbrace 1, \dots, 2r\rbrace $ let $ S_{j}\in \PSL_{2}(\RR)  $ be the isometry which sends $ \partial D_{j} $ to $ \partial D_{j+r} $ and maps the exterior of $ D_{j} $ to the interior of $ D_{j+r} $. We assume that the indices are defined cyclically, so that $ S_{j+2r} = S_{j} $ and $ S_{j}^{-1} = S_{j+r} $.

\begin{figure}[h]
\centering
\includegraphics[scale=0.35]{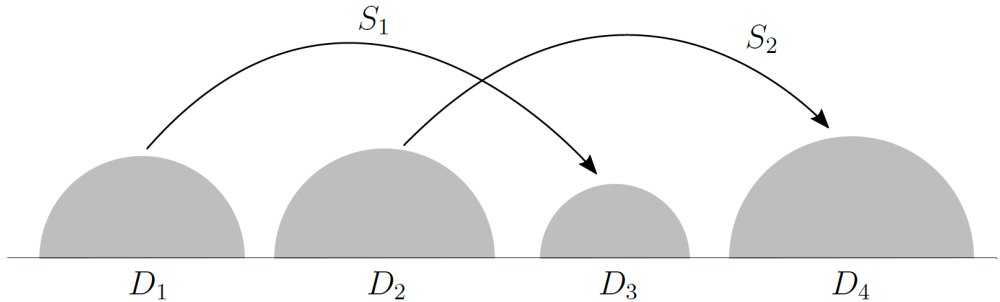}
\caption{Schottky disks and generators for $ r=2 $}
\end{figure}

A Fuchsian group $ \Gamma $ is a Schottky group if there exists a set of disks $ \{ D_{j} \} $ as above such that $ \Gamma $ is generated by the corresponding $ S_{j} $'s. (More precisely, this is a classical Fuchsian Schottky group, see Button \cite{Button}.) 

We now define the function space on which the associated transfer operator acts. Let $ \rho : \Gamma \to \mathrm{U}(V) $ be a finite-dimensional unitary representation on a complex Hilbert space $ V $ with inner product $ \langle\cdot, \cdot\rangle_{V} $.
Set
$$ U = \bigcup_{j=1}^{2r} D_{j}. $$
Consider the Hilbert space $ \mathcal{H}(U, V) $ of vector-valued holomorphic functions $ f : U \to V $ with
$$ \Vert f\Vert^{2}_{\mathcal{H}(U, \rho)} := \int_{U} \Vert f(z)\Vert_{V}^{2} dm(z) < \infty,  $$ 
where $ m $ is the Lebesgue measure and $ \Vert \cdot\Vert_{V} $ is the norm on $ V $ derived from the inner product $ \langle \cdot, \cdot\rangle_{V} $. On this space we can now define the ``$\rho$-twisted'' transfer operator $ \mathcal{L}_{s,\rho} $ by
$$ \mathcal{L}_{s,\rho}(f)(z) := \sum_{i\neq j+r} \left[  (S_{i}^{-1})'(z) \right]^{s} \rho(S_{i})f(S_{i}^{-1}z) \text{ if } z\in D_{j}. $$
Note that $ (S_{i}^{-1})'(z) $ is positive for $ z\in \mathbb{R} $, so the power $ \left[  (S_{i}^{-1})'(z) \right]^{s} $ is well defined for $ z\in D_{j} $ by analytic continuation. We also point out that for all $ i\neq j+r $ the map $ S_{i}^{-1} : D_{j}\to D_{i} $ is a holomorphic contraction, since $ \overline{S_{i}^{-1}(D_{j})} \subset D_{i}. $ Consequently, $ \mathcal{L}_{s,\rho} $ is a compact \textit{trace class} operator and therefore $ \mathcal{L}_{s,\rho} $ has a \textit{Fredholm determinant}. For a basic reference for the theory of Fredholm determinants on Hilbert spaces, see \cite{Lidskii}.

The connection between vector-valued transfer operators and L-functions is given by the following result.

\begin{prop}\label{FredholIdentity} For all $ \R s \gg 0 $, we have the identity
$$ L_{\Gamma}(s, \rho) = \det(1-\mathcal{L}_{s,\rho}). $$
In particular, $ L_{\Gamma}(s,\rho) $ extends to an entire function.
\end{prop}

There is a long history of Fredholm determinant representations for the Selberg zeta functions that goes back to Ruelle  \cite{Ruelle}, Fried \cite{Fried}, Pollicott \cite{Pollicott} and Mayer \cite{Mayer}, Chang-Mayer \cite{ChangMayer} for the modular surface. Later M\"oller-Pohl \cite{Pohl2013} and Pohl \cite{Pohl2015}, \cite{Pohl2016} extended the thermodynamic formalism approach to a large class of hyperbolic orbifolds with cusps, which for instance includes all Hecke triangle orbifolds. 

The L-functions we introduced in Section \ref{L-Functions and Artin-Takagi Formula} are often referred to as ``twisted Selberg functions''. Fredholm determinant representations of twisted Selberg zeta functions recently appeared in the work of Fedosova and Pohl \cite{PF}. Fredholm determinant representations of L-Functions associated to congruence Schottky groups recently appeared in the work of Jakobson and Naud \cite{JN}. A proof of Proposition \ref{FredholIdentity} can be found in \cite{PF} and in \cite{JN}.

\subsection{Proof of Theorem 2}

Let $ \Gamma $ be a Schottky group and set $ X = \Gamma\setminus \HH  $. Schottky groups are free (see for instance \cite[Lemma~15.2]{Borthwick}), so that we may assume that $ \Gamma $ is isomorphic to a free group of rank $ r $ for some positive integer $ r $. Then the homology group $ H_{1}(X, \ZZ) = \Gamma/[\Gamma, \Gamma] $ is isomorphic to $ \ZZ^{r} $ and its torsion group is trivial. Let $ [\gamma] $ denote the image of $ \gamma\in \Gamma $ under the natural projection $ \Gamma\to H_{1}(X, \ZZ).$ Recall that by fixing an isomorphism $ H_{1}(X, \ZZ) \cong \ZZ^{r} $, we may regard $ [\gamma] $ as a vector in $ \ZZ^{r} $. For a vector $ \theta\in \mathbb{R}^{r} $ let $ \chi_{\theta} :\Gamma\to S^{1} $ denote the character given by
$$ \chi_{\theta}(\gamma) := e(\langle \theta,[\gamma] \rangle), $$ 
where $ e(t) := e^{2\pi i t} $ and $ \langle \theta, v \rangle := \theta_{1}v_{1} + \cdots + \theta_{1}v_{1} $ denotes the standard inner product on $ \mathbb{R}^{r}. $

Set $ G_{N}:= \Gamma/\Gamma_{N}$ and recall that $ G_{N} $ is isomorphic to the abelian group $ (\ZZ/N\ZZ)^{r} $. We can identify $ (\ZZ/N\ZZ)^{r} $ with $ \{ 0, \dots, N-1\}^{r}\subset \RR^{r} $ as sets. Thus the set $ \widehat{G}_{N} $ of irreducible finite-dimensional representations of $ G_{N} $ is given by the set of characters $ \chi_{\frac{1}{N}a} $ with $ a\in \{ 0, \dots, N-1\}^{r}. $ 

For notational convenience, we will write $ L(s,\theta) := L(s, \chi_{\theta}) $ and $ \mathcal{L}_{s, \theta} := \mathcal{L}_{s, \chi_{\theta}}. $ By Proposition \ref{Analogue of the Artin-Takagi Formula} we have the factorization formula
\begin{equation}\label{ArtTak}
 Z_{\Gamma_{N}}(s) = \prod_{a\in \{ 0, \dots, N-1\}^{r}} L_{\Gamma}\left( s, \frac{1}{N}a\right).
\end{equation}

Let $ \Vert \theta\Vert_{1} = \vert \theta_{1}\vert + \cdots + \vert \theta_{r}\vert $ denote the 1-norm on $ \RR^{r}. $

\begin{lem}\label{vanishing property}
There exists $ \varepsilon_{1}>0 $ such that the following statement holds.
For all $ 0 <  \varepsilon < \varepsilon_{1} $ there exists a constant $ c_{1}>0 $ such that for all $ \theta\in \RR^{r} $, $ \Vert \theta\Vert_{1} < c_{1} $ implies that $ L_{\Gamma}(s,\theta) $ vanishes at some point $ s $ with $ \vert s-\delta\vert < \varepsilon. $
\end{lem}

\begin{proof}
Note that the map $ (s,\theta)\mapsto \mathcal{L}_{s, \theta} $ defines a continuously differentible map from $ \CC\times\RR^{r} $ into the trace class operators on the Hilbert space $ \mathcal{H}(U) $ of holomorphic $ L^{2} $-functions $ f:U\to \CC $. Hence, $ (s,\theta)\mapsto L_{\Gamma}(s,\theta)=\det(1-\mathcal{L}_{s, \theta}) $ is a continuously differentiable map. Since $ s = \delta $ is a simple zero of $ Z_{\Gamma}(s) $, 
$$ \frac{\partial}{\partial s} L_{\Gamma}(\delta, 0) = Z_{\Gamma}'(\delta) \neq 0. $$
By the (analytic) Inverse Function Theorem there exists a non-empty neighbourhood $ B_{\varepsilon_{1}}(0)=\{\theta\in \RR^{r} : \Vert \theta\Vert_{1} < \varepsilon_{1} \} $ of $ 0 $ and a continuous function $ s : B_{\varepsilon_{1}}(0)\to \CC $ with $ s(0) = \delta $ and $ L_{\Gamma}(s(\theta), \theta) = 0. $ Continuity of $ s :  B_{\varepsilon_{1}}(0)\to \CC $ now implies the statement of Lemma \ref{vanishing property}. 
\end{proof}

We can now prove Theorem \ref{main theorem}.

\begin{proof}[Proof of Theorem \ref{main theorem}]

Let $ \varepsilon_{1} $ be as in Lemma \ref{vanishing property} and define $ \varepsilon_{0}:= \min\{ \varepsilon_{1}, \frac{\delta}{2} \} $. Fix $ \varepsilon\in (0, \varepsilon_{0}) $. By Lemma \ref{vanishing property} there exists a constant $ c_{1} > 0 $ such that for all $ \theta \in \mathbb{R}^{r} $, $ \Vert \theta\Vert_{1} < c_{1} $ implies $ L_{\Gamma}(\zeta,\theta) = 0 $ for some $ \zeta\in \CC $ with $ \vert \zeta-\delta\vert < \varepsilon $. Hence, by (\ref{ArtTak}) we have
\begin{equation}\label{counting}
\# \{ \zeta\in \CC : Z_{\Gamma_{N}}(\zeta)=0,\quad \vert \delta-\zeta\vert < \varepsilon\} \geq \# \left\{ a\in \{ 0, \dots, N-1\}^{r}  : \left\Vert \frac{1}{N}a \right\Vert_{1} < c_{1} \right\},
\end{equation}
where the zeros of $ Z_{\Gamma_{N}} $ are counted with multiplicities.
The right hand side of (\ref{counting}) is equal to 
$$ \#\left\{ (a_{1},\dots, a_{r})\in \{ 0, \dots, N-1 \}^{r} : a_{1}+\cdots +a_{r} < c_{1}N \right\}, $$
which can be estimated from below by 
$$ \#\left\{ (a_{1},\dots, a_{r})\in \{ 0, \dots, N-1 \}^{r} : 0\leq a_{i}< \frac{c_{1}N}{r} \right\} \geq (c_{0}N)^{r}, $$ 
where $ c_{0}:= \min\{1, \frac{c_{1}}{r} \} $. 
By the result of Borthwick, Judge and Perry \cite{BJP}, the zero set of $ Z_{\Gamma_{N}} $ splits into ``topological zeros'' at the points $ s\in \{ 0, -1, -2, \dots\} $, and the set $ \mathcal{R}_{X_{N}} $ of resonances of $ X_{N} $, counted with multiplicities. Hence, by our choice of $ \varepsilon_{0} $ we have 
\begin{equation}
\# \{ \zeta\in \mathcal{R}_{X_{N}} : \vert \delta-\zeta\vert < \varepsilon\} \geq \left( c_{0}N\right)^{r} = c_{0}^{r}[\Gamma : \Gamma_{N}],
\end{equation}
where the resonances are counted with multiplicities. This completes the proof of the Theorem \ref{main theorem}.
\end{proof}

\bibliographystyle{plain}
\bibliography{mybib}

\end{document}